\documentclass[11pt]{article}

\usepackage[T1]{fontenc}
\usepackage[utf8]{inputenc}
\usepackage{lmodern}
\usepackage{geometry}
\usepackage{amsmath,amssymb,amsthm,mathtools}
\usepackage{mathrsfs}
\usepackage{hyperref}
\usepackage{xcolor}
\usepackage{enumitem}
\usepackage{microtype}
\usepackage{stmaryrd}

\usepackage{authblk}

\hypersetup{
  colorlinks=true,
  linkcolor=blue!50!black,
  citecolor=blue!50!black,
  urlcolor=blue!50!black
}

\setlist[enumerate]{itemsep=0.2em, topsep=0.3em}

\newtheorem{definition}{Definition}[section]
\newtheorem{proposition}[definition]{Proposition}
\newtheorem{theorem}[definition]{Theorem}

\newtheorem{remark}[definition]{Remark}
\newtheorem{example}[definition]{Example}

\newcommand{\Sh}{\mathbf{Sh}}
\newcommand{\Set}{\mathbf{Set}}
\newcommand{\Hom}{\mathrm{Hom}}
\newcommand{\id}{\mathrm{id}}

\newcommand{\Ob}{\mathrm{Ob}}

\newcommand{\op}{\mathrm{op}}

\newcommand{\Grp}{\mathbf{Grp}}

\newcommand{\true}{\mathrm{true}}
\newcommand{\E}{\mathcal E}
\newcommand{\F}{\mathcal F}
\newcommand{\G}{\mathcal G}
\newcommand{\Pcal}{\mathcal P}

\title{\bf From Torsors to Topoi: An Introduction with a View Toward $\Sigma$-Protocols in Cryptography}

\author{\Large Takao Inou\'{e}}
\affil{\large Faculty of Informatics, Yamato University, \\ Osaka, Japan\footnote{Email: inoue.takao@yamato-u.ac.jp; \\ Personal Email: takaoapple@gmail.com \\ [I prefer my personal email address for correspondence.]}}
\date{March 17, 2026}

\begin{document}
\maketitle

\begin{abstract}
This paper provides a preparatory introduction to sheaves and topoi,
written as a conceptual continuation of the author's earlier introduction to torsors
and as preparatory background for the author's arXiv paper
\emph{Grothendieck Topologies and Sheaf-Theoretic Foundations of Cryptographic Security:\ Attacker Models and $\Sigma$-Protocols as the First Step}~\cite{InoueSecurity}.
Rather than attempting an encyclopedic survey of all of topos theory,
the exposition develops those parts of the subject that are most relevant
for passing from torsor-based local-to-global reasoning
to sheaf-theoretic and topos-theoretic reasoning:
Grothendieck topologies, sheaves, torsors over a site, descent,
sheaf topoi, elementary topoi, Cartesian closed structure,
subobject classifiers, and internal logic.
The goal is not merely motivational.
We try to develop enough genuine topos theory that the reader can understand,
not only heuristically but structurally,
why the later cryptographic framework of~\cite{InoueSecurity}
uses Grothendieck topologies and sheaf-theoretic language.
To make the note more self-contained,
we also include substantial appendices on basic category theory,
Yoneda's lemma, limits and colimits, equalizers and coequalizers,
Kan extensions, the relation between internal logic and intuitionistic logic,
and exercises with solutions.
In the final part, we explain how these ideas prepare the ground
for a conceptual understanding of $\Sigma$-protocols,
especially in connection with local consistency, simulability,
and the passage from compatible local data to global structure.
\end{abstract}

\bf keywords\rm : topos theory, sheaf theory, Grothendieck topology, torsors, descent, internal logic, 
intuitionistic logic, category theory, Yoneda lemma, $\Sigma$-protocols, cryptographic security, 
smulability, local-to-global principle
\medskip

\bf MSC2020\rm : 18F10, 18F20, 18N10, 03G30, 94A60

\tableofcontents

\section{Introduction}

This paper is intended as the middle term of a three-step line of exposition.
The first step is the author's earlier note on torsors~\cite{InoueTorsors},
where free transitive group actions, transport, local triviality,
and cocycle-based gluing were emphasized.
The third step is the author's later paper~\cite{InoueSecurity},
which proposes a sheaf-theoretic approach to cryptographic security
based on Grothendieck topologies, attacker models, and $\Sigma$-protocols.
The present paper is written to fill the conceptual gap between these two works.
Its purpose is to explain, in a mathematically serious way,
why the route from torsors to topoi is natural,
and why a reader who wishes to understand~\cite{InoueSecurity}
should first understand at least the basic architecture of topos theory.

There are two reasons why a superficial bridge is not enough.
First, torsor theory already contains a local-to-global philosophy:
a torsor is locally indistinguishable from the acting group,
yet globally it may fail to admit any distinguished origin.
This leads naturally to local trivializations,
transition functions, compatibility on overlaps, and descent.
Second, the cryptographic framework of~\cite{InoueSecurity}
does not merely borrow the language of sheaves as a metaphor.
It treats Grothendieck topologies, local observability,
and gluing of information as structural ingredients.
Accordingly, the reader must understand not only what a sheaf is,
but also what a topos is, how internal logic works,
and why such a setting supports disciplined local reasoning.

For this reason, the present note is more substantial than a short conceptual bridge.
It is still introductory in style, but it aims to provide enough genuine topos theory
that the later cryptographic applications can be understood from the inside.
We therefore discuss not only sheaves and torsors over a site,
but also the passage from sheaf categories to topoi,
the elementary definition of a topos,
the role of exponentials and subobject classifiers,
and the basic idea of internal logic.
The discussion is selective rather than encyclopedic,
but it is meant to be mathematically real.

A guiding principle throughout the paper is that
the move from torsors to topoi is not a change of subject.
Rather, it is a conceptual enlargement of structures already present in torsor theory.
The lack of a canonical origin in a torsor,
its recovery from local trivializations,
and the cocycle relations on overlaps
all point toward a setting in which local data and gluing are primary.
Sheaf theory is the first systematic language for such phenomena.
Topos theory then enlarges this language into a full mathematical universe
with its own internal notion of truth, existence, and function.
From that perspective,
understanding topoi is part of understanding what torsors were already trying to say.

The paper is organized as follows.
Section~\ref{sec:sites-sheaves} reviews Grothendieck topologies and sheaves.
Section~\ref{sec:torsors-descent} revisits torsors as sheaf-theoretic and site-theoretic objects,
with emphasis on cocycles and descent.
Section~\ref{sec:sheaf-topoi} explains why sheaf categories lead naturally to topoi.
Section~\ref{sec:elementary-topoi} introduces elementary topoi.
Section~\ref{sec:internal-logic} discusses subobject classifiers and internal logic in greater detail.
Section~\ref{sec:torsors-in-topos} places torsors inside a topos.
Section~\ref{sec:practice} provides a practice section before the cryptographic transition.
Section~\ref{sec:sigma} then explains why this entire progression matters
for the conceptual study of $\Sigma$-protocols and for~\cite{InoueSecurity}.
Section~\ref{sec:conclusion} concludes the paper by summarizing the passage from torsors to topoi
and its relevance to the later cryptographic framework.
The appendices provide substantial supporting material: a brief introduction to category theory,
Yoneda's lemma, Cartesian closed categories, limits and colimits, equalizers and coequalizers,
Kan extensions, the relation between internal logic and intuitionistic logic,
and solutions to the exercises.

\section{Sites, Grothendieck topologies, and sheaves}\label{sec:sites-sheaves}

The first serious step from torsors to topoi is the passage from ordinary open covers
to Grothendieck topologies.
This makes it possible to speak of locality and gluing
not only on topological spaces,
but on much more general categories.

\begin{definition}
Let $\mathcal C$ be a category.
A \emph{sieve} $S$ on an object $U\in\Ob(\mathcal C)$
is a collection of morphisms with codomain $U$
that is closed under precomposition:
if $f:V\to U$ lies in $S$ and $g:W\to V$ is any morphism,
then $f\circ g:W\to U$ also lies in $S$.
\end{definition}

\begin{definition}
A \emph{Grothendieck topology} $J$ on a category $\mathcal C$
assigns to each object $U$ a collection $J(U)$ of sieves on $U$,
called \emph{covering sieves},
such that:
\begin{enumerate}[label=(\roman*)]
\item the maximal sieve belongs to $J(U)$;
\item if $S\in J(U)$ and $f:V\to U$ is any morphism,
then the pullback sieve $f^*S$ belongs to $J(V)$;
\item if $S$ is a sieve on $U$ and there exists $R\in J(U)$
such that for every $f:V\to U$ in $R$, the pullback sieve $f^*S$ belongs to $J(V)$,
then $S\in J(U)$.
\end{enumerate}
The pair $(\mathcal C,J)$ is called a \emph{site}.
\end{definition}

These axioms abstract the familiar behavior of open covers.
Condition (ii) says that coverings pull back.
Condition (iii) is a transitivity condition:
a family that becomes covering after refinement by a cover was already covering.

\begin{example}
Let $X$ be a topological space and let $\mathcal O(X)$ be its category of open sets,
with morphisms given by inclusions.
A family $\{U_i\subseteq U\}$ is covering if $U=\bigcup_i U_i$.
The sieves generated by such families define a Grothendieck topology on $\mathcal O(X)$.
The associated sheaf theory is the usual sheaf theory on $X$.
\end{example}

\begin{definition}
Let $(\mathcal C,J)$ be a site.
A presheaf of sets on $\mathcal C$ is a contravariant functor
$\F:\mathcal C^\op\to\Set$.
It is a \emph{sheaf} if for every object $U$ and every covering sieve $S\in J(U)$,
the natural map
\[
\F(U)\longrightarrow \Hom(S,\F)
\]
is a bijection.
Equivalently, sections over $U$ are uniquely determined by compatible sections
on every cover of $U$.
\end{definition}

In elementary terms,
a sheaf is a presheaf for which local compatible data glue uniquely.
The modern form of the definition matters because it shows that
``locality'' depends only on the chosen Grothendieck topology,
not on any ambient notion of point-set space.
This flexibility is exactly what will later matter in cryptographic applications,
where the ``local'' pieces are not literally open subsets of a topological space,
but information states organized by a coverage relation~\cite{InoueSecurity}.

\begin{proposition}
Let $(\mathcal C,J)$ be a site.
If $\F$ is a sheaf and $\{U_i\to U\}$ is a covering family,
then any family of local sections $s_i\in\F(U_i)$
that agree on all overlaps $U_i\times_U U_j$
glues to a unique global section $s\in\F(U)$.
\end{proposition}

\begin{proof}[Proof sketch]
The family $\{U_i\to U\}$ generates a covering sieve on $U$.
Compatibility on the pullback overlaps says precisely that the induced map
from the sieve to $\F$ is natural.
By the sheaf property, such a natural transformation comes from a unique section of $\F(U)$.
\end{proof}

Thus sheaf theory is not simply a technical add-on.
It is the formal expression of the principle that local information,
if compatible, should reconstruct global information.
That is the precise direction in which torsor theory already points.

\section{Torsors, cocycles, and descent on a site}\label{sec:torsors-descent}

In the earlier torsor note~\cite{InoueTorsors}, torsors were introduced
through free transitive group actions.
On a site, this idea persists,
but local triviality now replaces the existence of an honest global basepoint.

\begin{definition}
Let $(\mathcal C,J)$ be a site and let $\G$ be a sheaf of groups on it.
A \emph{$\G$-torsor} is a sheaf $\Pcal$ with a right action of $\G$
such that:
\begin{enumerate}[label=(\roman*)]
\item $\Pcal$ is locally nonempty:
there exists a covering family $\{U_i\to U\}$ such that each $\Pcal(U_i)$ is nonempty;
\item the action is locally free and transitive:
for every object $U$ and every $p,q\in \Pcal(U)$,
there exists a unique $g\in\G(U)$ such that $q=p\cdot g$.
\end{enumerate}
\end{definition}

This definition captures the same geometry as in the classical case.
A torsor looks like the acting group once a local section has been chosen,
but no such choice is canonical globally.
The resulting failure of global triviality is measured by cocycles.

\begin{proposition}\label{prop:cocycle-torsor}
Let $\Pcal$ be a $\G$-torsor and let $\{U_i\to U\}$ be a covering family
with chosen local sections $s_i\in\Pcal(U_i)$.
Then on each overlap $U_{ij}:=U_i\times_U U_j$
there exists a unique element
\[
g_{ij}\in\G(U_{ij})
\]
such that
\[
s_j=s_i\cdot g_{ij}
\quad \text{on }U_{ij},
\]
and these elements satisfy the cocycle condition
\[
g_{ij}g_{jk}=g_{ik}
\quad \text{on }U_i\times_U U_j\times_U U_k.
\]
\end{proposition}

\begin{proof}[Proof sketch]
Uniqueness and existence of each $g_{ij}$ come from free transitivity.
Evaluating $s_k$ through $s_i$ either directly or via $s_j$
gives the cocycle identity on triple overlaps.
\end{proof}

\begin{remark}
The cocycle of Proposition~\ref{prop:cocycle-torsor}
does not depend on the chosen local trivializations in an absolute sense,
but only up to the expected change-of-trivialization relations.
Thus the torsor is encoded by descent data rather than by a preferred global point.
\end{remark}

This is the first place where descent becomes unavoidable.
A descent datum is, roughly speaking,
local data together with compatibility on overlaps,
subject to higher coherence on multiple overlaps.
Torsors are among the cleanest elementary examples of this principle.

\begin{definition}
Let $\{U_i\to U\}$ be a covering family in a site.
A \emph{descent datum} for presheaves or structured objects over the cover
consists of local objects on each $U_i$
together with isomorphisms on overlaps
satisfying a cocycle condition on triple overlaps.
\end{definition}

\begin{theorem}[Descent principle for torsors, informal form]
A $\G$-torsor over $U$ may be reconstructed from local trivial torsors on a cover of $U$
together with a cocycle satisfying the compatibility condition on triple overlaps.
\end{theorem}

\begin{proof}[Proof sketch]
One starts with the trivial torsors $\G|_{U_i}$
and uses the cocycle to identify them on overlaps.
The sheaf condition ensures that the glued object exists,
and the cocycle identity guarantees associativity of the gluing.
\end{proof}

In this sense, torsor theory already lives naturally on a site.
The later move to topoi does not discard this perspective;
it gives it a larger ambient universe.

\section{From sheaf categories to topoi}\label{sec:sheaf-topoi}

For a site $(\mathcal C,J)$,
the category $\Sh(\mathcal C,J)$ of sheaves of sets is not only a convenient receptacle for glued data.
It is itself a highly structured mathematical world.
This is the first appearance of the word \emph{topos}.

\begin{definition}
A \emph{Grothendieck topos} is a category equivalent to $\Sh(\mathcal C,J)$
for some site $(\mathcal C,J)$.
\end{definition}

The motivating example is $\Sh(X)$ for a topological space $X$.
But the definition is broader:
a Grothendieck topos is a category of sheaves on some generalized notion of space.
It should therefore be thought of both as
\begin{enumerate}[label=(\alph*)]
\item a category of locally varying sets, and
\item a generalized space encoded by the way those sets vary and glue.
\end{enumerate}

\begin{proposition}
For every site $(\mathcal C,J)$, the category $\Sh(\mathcal C,J)$ has finite limits.
Moreover, the inclusion $\Sh(\mathcal C,J)\hookrightarrow \Set^{\mathcal C^\op}$
preserves finite limits.
\end{proposition}

\begin{proof}[Proof sketch]
Finite limits of presheaves are computed pointwise.
The sheaf condition is stable under finite limits,
so pointwise finite limits of sheaves are again sheaves.
\end{proof}

This already gives a great deal of structure.
Products, equalizers, terminal objects,
and many familiar set-like constructions exist internally.
But a topos has even more:
roughly speaking, it also has function objects and an internal object of truth values.
That is what makes it a genuine replacement for a universe of sets.

\begin{remark}
The passage from a space $X$ to the topos $\Sh(X)$
may be read in two opposite ways.
One may say that $\Sh(X)$ is simply the category of sheaves on $X$.
But one may also say that $\Sh(X)$ \emph{is} the mathematically meaningful avatar of $X$,
since it remembers how objects vary locally over $X$ and how they glue.
This second viewpoint is one of the conceptual sources of modern topos theory.
\end{remark}

\section{Elementary topoi}\label{sec:elementary-topoi}

Grothendieck topoi arise from sites,
but there is also an intrinsic, axiomatic notion of topos.
This is the notion of an \emph{elementary topos}.
It isolates the structural features that make a category resemble $\Set$.

\begin{definition}
An \emph{elementary topos} is a category $\E$
that has
\begin{enumerate}[label=(\roman*)]
\item finite limits,
\item exponentials, and
\item a subobject classifier.
\end{enumerate}
\end{definition}

Let us briefly recall these ingredients.
Finite limits provide terminal objects, products, and equalizers.
Exponentials mean that for objects $A,B\in\E$,
there is an object $B^A$ representing morphisms from $A$ to $B$.
Thus one can speak of an internal function object.
The third ingredient, the subobject classifier,
is one of the most characteristic features of topos theory.

\begin{definition}
A \emph{subobject classifier} in a category $\E$ with finite limits
is an object $\Omega$ together with a morphism
\[
\true:1\to \Omega
\]
from the terminal object,
such that for every monomorphism $m:A\hookrightarrow X$
there exists a unique morphism
\[
\chi_m:X\to \Omega
\]
for which the square
\[
\begin{matrix}
A & \longrightarrow & 1 \\
m\,\downarrow && \downarrow\,\true \\
X & \xrightarrow{\chi_m} & \Omega
\end{matrix}
\]
is a pullback.
\end{definition}

In $\Set$, the subobject classifier is the two-element set $\{0,1\}$.
A subset $A\subseteq X$ is classified by its characteristic function
$\chi_A:X\to\{0,1\}$.
In a general topos, $\Omega$ plays the role of an object of truth values,
but its truth values need not be merely Boolean and global.
This is precisely where local logic enters.

\begin{example}
For a topological space $X$,
the topos $\Sh(X)$ is an elementary topos.
Its subobject classifier $\Omega$ is the sheaf of open subsets:
for each open set $U\subseteq X$,
$\Omega(U)$ is the set of open subsets of $U$.
Thus truth over $U$ is represented not by a single global bit,
but by an open region of $U$ on which a statement holds.
\end{example}

\begin{proposition}
Every Grothendieck topos is an elementary topos.
\end{proposition}

\begin{proof}[Proof sketch]
For $\Sh(\mathcal C,J)$,
finite limits are inherited from the presheaf category,
exponentials exist because sheafification preserves the needed structure,
and a subobject classifier can be constructed by sheafifying the presheaf of locally covering sieves.
Full proofs may be found in standard references~\cite{MacLaneMoerdijk,JohnstoneSketches}.
\end{proof}

Thus the passage from sites to sheaves has brought us
not merely to a category of glued objects,
but to a set-like universe with its own internal structure.

\section{Subobject classifiers and internal logic}\label{sec:internal-logic}

One of the deepest insights of topos theory is that
a topos carries its own internal logic.
This logic is not imposed from outside;
it arises from the category's structural features,
especially the subobject classifier.

In ordinary set theory,
a predicate on a set $X$ is represented by a subset of $X$,
or equivalently by a characteristic function $X\to\{0,1\}$.
In a topos,
a predicate on an object $X$ is represented by a subobject of $X$,
and hence by a classifying arrow $X\to\Omega$.
Thus $\Omega$ serves as an object of truth values.
The difference is that these truth values may vary locally.

\begin{remark}
In $\Sh(X)$, the truth value of a statement over an open set $U$
is generally not simply ``true'' or ``false''.
It is an open subset of $U$ on which the statement holds.
Hence truth is local and geometric.
This is one of the central reasons topoi are so well adapted
to local-to-global reasoning.
\end{remark}

The slogan that a topos has an internal logic means at least three things.
First, subobjects behave like predicates.
Second, $\Omega$ behaves like an internal algebra of truth values.
Third, quantifiers can be interpreted by categorical adjunctions,
so that statements about existence and universality can be expressed internally.
Even when one does not formalize everything,
these three points already give a serious mathematical content to the phrase
``reasoning inside a topos.''

\begin{example}[Truth values in a sheaf topos]
Let $X$ be a topological space and let $\Omega\in\Sh(X)$ be the subobject classifier.
For an open set $U\subseteq X$, the set $\Omega(U)$ is the set of open subsets of $U$.
If $A\hookrightarrow F$ is a subobject of a sheaf $F$,
then the classifying map $\chi_A:F\to\Omega$
assigns to a section $s\in F(U)$
the largest open subset of $U$ on which $s$ belongs to $A$.
Thus predicates do not return a global bit;
they return a region of validity.
\end{example}

\begin{example}[A concrete local predicate]
Take the sheaf $C^0(-,\mathbb R)$ of continuous real-valued functions on a space $X$.
Let $P\hookrightarrow C^0(-,\mathbb R)$ be the subobject consisting of strictly positive functions.
Then for $f\in C^0(U,\mathbb R)$,
the truth value $\chi_P(f)\in\Omega(U)$
is the open subset
\[
\{x\in U\mid f(x)>0\}.
\]
So the internal statement ``$f>0$'' is not simply true or false on $U$;
it is true precisely on the open region where positivity holds.
This is an elementary but very instructive model of local truth.
\end{example}

One may push this much further.
In an elementary topos,
one can interpret conjunction, implication,
quantification, equality, and existence.
This yields a higher-order intuitionistic logic internal to the topos.
For the present paper, the full formal apparatus is unnecessary,
but the following heuristic is indispensable:
\begin{quote}
Statements in a topos are evaluated relative to the local context,
and existential claims amount to the existence of compatible local data.
\end{quote}

This becomes especially transparent in Kripke--Joyal semantics.
In a sheaf topos,
to say that an open set $U$ forces an existential statement
means that $U$ can be covered by smaller opens on which witnesses exist,
with the expected compatibility conditions.
Thus the logical meaning of existence is already sheaf-theoretic.

\begin{proposition}[Kripke--Joyal heuristic for existence]
Let $\Sh(X)$ be a sheaf topos and let $\varphi(x)$ be an internal statement.
Then the assertion that ``$U\Vdash \exists x\,\varphi(x)$''
should be read as follows:
there exists an open cover $U=\bigcup_i U_i$
and sections $x_i$ over $U_i$
such that each $U_i$ forces $\varphi(x_i)$,
with compatibility on overlaps whenever required.
\end{proposition}

\begin{proposition}[Kripke--Joyal heuristic for universal statements]
Let $\Sh(X)$ be a sheaf topos and let $\varphi(x)$ be an internal statement about sections of a sheaf $F$.
Then ``$U\Vdash \forall x\,\varphi(x)$'' means that
for every open subset $V\subseteq U$ and every section $x\in F(V)$,
one has $V\Vdash \varphi(x)$.
Thus universal truth is stable under restriction to smaller opens.
\end{proposition}

\begin{remark}
This difference between existential and universal statements is crucial.
Existential truth may have to be established only after passing to a cover,
whereas universal truth must survive every further localization.
That asymmetry is one of the signatures of intuitionistic and geometric reasoning.
\end{remark}

\begin{example}[Local existence versus global existence]
Let $X=S^1$ and let $F$ be the sheaf of continuous real-valued functions.
Consider the statement
\[
\exists g\in F\; (g^2=f)
\]
for a fixed section $f\in F(X)$.
If $f$ is everywhere positive, then this statement is globally true.
If $f$ changes sign, it is globally false.
But even when no global square root exists,
one may still have local square roots on a suitable open cover.
In internal language,
this means that the existential statement can be forced locally without being forced globally.
\end{example}

\begin{example}[Implication as local stability]
For subobjects $A,B\hookrightarrow F$ in a topos,
the implication $A\Rightarrow B$ is another subobject of $F$,
characterized internally by the property that a section belongs to it
exactly where membership in $A$ locally forces membership in $B$.
In a sheaf topos, implication is therefore not a purely pointwise operation;
it records local stability of entailment.
\end{example}

\begin{remark}
The truth values in an elementary topos form internally a Heyting algebra rather than, in general, a Boolean algebra.
For this reason, excluded middle need not hold.
The logical meaning of this fact is discussed in Appendix~\ref{app:intuitionistic}.
For now, the important point is that local truth is usually subtler than classical binary truth.
\end{remark}

\begin{proposition}[Internal logic as organized local reasoning]
Let $\E$ be a Grothendieck topos.
Then the interpretation of predicates by subobjects,
of truth values by $\Omega$,
and of quantifiers by the internal adjoint structure of the topos
organizes reasoning in such a way that local compatibility data can be expressed and manipulated internally.
\end{proposition}

\begin{proof}[Proof sketch]
Predicates are represented by monomorphisms and therefore by classifying arrows into $\Omega$.
Products and exponentials support conjunction and function spaces.
The existence of appropriate adjoints for pullback along projections supplies the categorical content of quantification.
In a Grothendieck topos, Kripke--Joyal semantics makes these constructions explicitly local.
\end{proof}

\begin{remark}
This does not mean that all reasoning in a topos is merely local patchwork.
Rather, it means that global truth is mediated by local truth and gluing.
The point is not the denial of global structure,
but the disciplined reconstruction of global structure from local data.
\end{remark}

This perspective is crucial for the intended cryptographic application.
The later paper~\cite{InoueSecurity}
uses Grothendieck topologies and sheaf-like organization
precisely because attacker knowledge, observability,
and compatibility of information are not simply global yes/no matters.
The internal logic of a topos is therefore not an ornamental abstraction;
it is a model for local reasoning with controlled gluing.

\section{Practice: working with sheaves, topoi, and internal logic}\label{sec:practice}

The following exercises are intended to help the reader consolidate the main ideas of the paper before passing to the explicitly cryptographic discussion.
They are not logically necessary for the later sections,
but they are pedagogically useful.
Solutions are collected in Appendix~\ref{app:solutions}.

\begin{enumerate}[label=\textbf{Exercise \arabic*.}, leftmargin=*]
\item Let $X$ be a topological space and let $\Omega\in\Sh(X)$ be the subobject classifier.
Show that for each open set $U\subseteq X$, the set $\Omega(U)$ may be identified with the set of open subsets of $U$.
Explain why this means that truth values in $\Sh(X)$ are local.

\item Let $\G$ be a sheaf of groups on a site $(\mathcal C,J)$ and let $\Pcal$ be a $\G$-torsor.
Choose local sections on a covering family and derive the corresponding cocycle $g_{ij}$.
Verify formally that the cocycle condition on triple overlaps expresses compatibility of transport.

\item In a sheaf topos, explain in words the difference between the meanings of
\[
U\Vdash \exists x\,\varphi(x)
\quad\text{and}\quad
U\Vdash \forall x\,\varphi(x).
\]
Why is the first statement allowed to pass to a cover while the second must hold after every restriction?

\item Let $F=C^0(-,\mathbb R)$ on a topological space $X$ and let $P\hookrightarrow F$ be the subobject of positive functions.
For a section $f\in F(U)$, describe concretely the classifying map $\chi_P(f)\in \Omega(U)$.
What does this say about the internal meaning of the predicate ``$f>0$''?

\item Let $\E$ be an elementary topos.
Explain why the data of finite limits, exponentials, and a subobject classifier make $\E$ resemble a universe of sets.
Which of these ingredients is most directly responsible for internal truth values?
\end{enumerate}

\section{Torsors and local symmetry inside a topos}\label{sec:torsors-in-topos}

We now return to torsors,
but from the more mature viewpoint provided by topos theory.
Once one is working in a topos $\E$,
it makes sense to speak of internal groups, internal actions,
and therefore internal torsors.

\begin{definition}
An \emph{internal group} in a category with finite products
is an object $G$ equipped with multiplication, unit, and inversion morphisms
satisfying the usual group diagrams internally.
\end{definition}

\begin{definition}
Let $G$ be an internal group in a topos $\E$.
A \emph{$G$-torsor in $\E$} is an object $P$ with a right $G$-action
such that internally:
\begin{enumerate}[label=(\roman*)]
\item $P$ is inhabited locally,
\item the action is free,
\item the action is transitive.
\end{enumerate}
Equivalently, the canonical map
\[
P\times G\longrightarrow P\times P,
\qquad (p,g)\mapsto (p,p\cdot g),
\]
is an isomorphism, and $P\to 1$ is an epimorphism.
\end{definition}

This definition shows clearly how torsors belong to the topos world.
The absence of a global origin is not a defect.
It is a legitimate internal form of symmetry.
What matters is not the existence of a preferred point,
but the existence of local points and the ability to transport among them.

\begin{proposition}
Let $(\mathcal C,J)$ be a site and let $\G$ be a sheaf of groups.
Then $\G$-torsors on the site are precisely torsors for the internal group $\G$
in the topos $\Sh(\mathcal C,J)$.
\end{proposition}

\begin{proof}[Proof sketch]
The site-theoretic definition of a $\G$-torsor is exactly the sheaf-theoretic version of local inhabitation together with free transitive action.
Inside the topos, these become the internal epimorphism condition and the isomorphism $P\times G\cong P\times P$.
The two definitions translate directly into one another.
\end{proof}

\begin{remark}
This proposition explains why the route taken in this paper is conceptually coherent.
The first torsor note~\cite{InoueTorsors} introduced torsors externally.
The present paper shows how the same idea is naturally internalized in a topos.
That is already enough to suggest why a later sheaf-theoretic account of cryptographic security~\cite{InoueSecurity}
should treat local knowledge and compatibility structurally rather than ad hoc.
\end{remark}

\section{Toward $\Sigma$-protocols and cryptographic security}\label{sec:sigma}

We now explain, at a conceptual level,
why the progression from torsors to topoi is relevant to $\Sigma$-protocols.
The point is not that a standard protocol transcript is literally a sheaf or a torsor in any naive sense.
The point is that the mathematics needed to understand
local consistency, simulability, and global knowledge
already has a natural language in sheaf and topos theory.

In the torsor note~\cite{InoueTorsors},
a recurring theme was the possibility of coherent transport without a distinguished origin.
In~\cite{InoueSecurity},
the analogous cryptographic theme is that local views, attacker observations,
and simulated pieces of evidence can carry structure
without there being a single naive global witness visible everywhere.
The relevant mathematics is therefore not merely the mathematics of global objects,
but the mathematics of locally available compatible objects.

This is where Grothendieck topologies matter.
A Grothendieck topology formalizes what counts as a covering family,
and hence what counts as a local view sufficient for gluing.
In the cryptographic setting of~\cite{InoueSecurity},
one studies information through attacker models and coverings generated by accessible observations.
Sheaves then encode consistency of information across such local observations.
The role of topos theory is to provide the ambient logic in which such local reasoning can be carried out coherently.

\begin{remark}
One may think heuristically of a simulator as providing local trivializations,
not of a geometric bundle, but of an information structure.
A successful simulation says that certain local observables can be produced compatibly.
The passage from these local pieces to a meaningful global statement
is therefore naturally reminiscent of descent.
This is not a theorem of protocol theory by itself,
but it is precisely the kind of structural analogy
that makes sheaf-theoretic and topos-theoretic language fruitful.
\end{remark}

The author's paper~\cite{InoueSecurity}
should be understood against this background.
That paper does not emerge suddenly from ordinary set-based protocol notation.
It emerges from a sustained local-to-global viewpoint:
first torsors, then sheaves and topoi, and only then cryptographic security.
Accordingly, the present paper and the earlier torsor note are meant as genuine preparation,
not as unrelated expository side remarks.
The intended line is
\[
\begin{aligned}
&\text{torsors and cocycles}
\longrightarrow
\text{sheaves and descent}
\\
&\qquad\longrightarrow
\text{topoi and internal logic}
\longrightarrow
\text{cryptographic local reasoning}.
\end{aligned}
\]

\section{Outlook}

Several directions remain open.
One may first strengthen the present introductory discussion
by giving a fuller account of geometric morphisms,
sites of definition,
and the relation between internal and external viewpoints.
Second, one may study torsors not merely in an ordinary sheaf topos,
but in more refined settings involving stacks or higher topoi.
Third, and most relevant for the author's current program,
one may try to make the cryptographic heuristics of Section~\ref{sec:sigma}
more precise by identifying exact sheaf-theoretic structures
attached to attacker models and classes of protocols.

For the moment, however, the aim of this paper is preparatory.
Together with the earlier torsor note~\cite{InoueTorsors},
it is intended to make the later paper~\cite{InoueSecurity}
conceptually approachable.
If the present discussion succeeds,
the reader should come away with the sense that topos theory is not an optional abstract superstructure,
but a natural enlargement of the local-to-global mathematics already present in torsors,
and therefore a serious candidate language for the structural study of cryptographic security.

\section{Conclusion}\label{sec:conclusion}

We have argued that the passage from torsors to topoi is not an accidental change of language,
but a structural enlargement of ideas already present in torsor theory.
Local triviality, compatibility on overlaps, cocycle data, and descent already point toward a world
in which local information and its gluing are primary.
Sheaf theory organizes this viewpoint systematically, and topos theory extends it further
into a setting with its own internal notions of truth, existence, and function.

For that reason, the present paper has pursued two goals at once.
On the one hand, it has served as an introduction to sites, sheaves, topoi, internal logic,
and related categorical tools.
On the other hand, it has been written as explicit preparation for the author's cryptographic paper~\cite{InoueSecurity},
in which Grothendieck topologies and sheaf-theoretic reasoning are used as genuine structural ingredients.
The earlier torsor note~\cite{InoueTorsors} supplied the first step of this line of thought;
the present paper develops the second step by making the topos-theoretic background more fully visible.

The appendices were included for a pedagogical reason.
A reader who is not yet fully comfortable with category theory, internal logic,
or limits and colimits should still be able to use this paper as a working preparation for later study.
In that sense, the paper is meant not only as a conceptual bridge, but also as a practical companion.
If it succeeds, the reader should be able to approach~\cite{InoueSecurity}
with a clearer understanding of why local reasoning, gluing, and sheaf-theoretic structure
belong naturally to the mathematical study of cryptographic security and $\Sigma$-protocols.

\bigskip

\noindent Takao Inou\'{e}

\noindent Faculty of Informatics

\noindent Yamato University

\noindent Katayama-cho 2-5-1, Suita, Osaka, 564-0082, Japan

\noindent inoue.takao@yamato-u.ac.jp
 
\noindent (Personal) takaoapple@gmail.com (I prefer my personal mail)
\bigskip

\appendix

\section{A brief introduction to category theory}\label{app:category}

This appendix is included for readers who are not yet comfortable with category theory
but would still like to follow the main text.
The aim is not to replace a full course,
but to provide enough language and technique
for basic work with sheaves, topoi, and internal logic.
In particular, we review categories, functors, natural transformations,
universal constructions, the Yoneda lemma,
and Cartesian closed categories.

\subsection{Categories, morphisms, and commutative diagrams}

A \emph{category} $\mathcal C$ consists of:
\begin{enumerate}[label=(\roman*)]
\item a class of objects $A,B,C,\dots$;
\item for each pair of objects $A,B$, a set $\Hom_{\mathcal C}(A,B)$ of morphisms $A\to B$;
\item for each triple $A,B,C$, a composition law
\[
\Hom_{\mathcal C}(B,C)\times \Hom_{\mathcal C}(A,B)\to \Hom_{\mathcal C}(A,C),
\qquad (g,f)\mapsto g\circ f;
\]
\item for each object $A$, an identity morphism $\id_A\colon A\to A$.
\end{enumerate}
These data satisfy associativity
\[
(h\circ g)\circ f = h\circ (g\circ f)
\]
and the identity laws
\[
\id_B\circ f = f,
\qquad
f\circ \id_A = f.
\]

Typical examples are the category $\Set$ of sets,
the category $\Grp$ of groups,
and the category $\Sh(X)$ of sheaves on a space $X$.
In all of these, one studies not only objects in isolation,
but also the maps between them and the way constructions behave functorially.

A diagram in a category is said to \emph{commute}
if every pair of morphism composites with the same source and target are equal.
For example,
\[
\begin{array}{c}
A \xrightarrow{f} B \\
\ \searrow_h \qquad \downarrow g \\
\qquad C
\end{array}
\]
commutes if and only if $g\circ f = h$.
Many categorical definitions are nothing but the statement that a certain diagram commutes
and is universal among all such diagrams.

\subsection{Functors and natural transformations}

A \emph{functor} $F\colon \mathcal C\to \mathcal D$
sends each object $A$ of $\mathcal C$ to an object $F(A)$ of $\mathcal D$
and each morphism $f\colon A\to B$ to a morphism
\[
F(f)\colon F(A)\to F(B)
\]
in such a way that
\[
F(g\circ f)=F(g)\circ F(f),
\qquad
F(\id_A)=\id_{F(A)}.
\]
Thus a functor preserves categorical structure.
For example, the power-set construction $X\mapsto \mathcal P(X)$ is a functor $\Set\to\Set$,
and the assignment $U\mapsto \F(U)$ for a sheaf $\F$ is a contravariant functor from open sets to sets.

If $F,G\colon \mathcal C\to \mathcal D$ are functors,
a \emph{natural transformation} $\eta\colon F\Rightarrow G$
is a family of morphisms
\[
\eta_A\colon F(A)\to G(A)
\]
indexed by objects $A$ of $\mathcal C$,
such that for every morphism $f\colon A\to B$,
the square
\[
\begin{array}{ccc}
F(A) & \xrightarrow{\eta_A} & G(A) \\
\downarrow F(f) & & \downarrow G(f) \\
F(B) & \xrightarrow{\eta_B} & G(B)
\end{array}
\]
commutes.
Naturality expresses the idea that the maps $\eta_A$ are compatible with all structure maps in sight.
This is one of the basic ways in which category theory formalizes ``the same construction in all contexts.''

\subsection{Universal properties}

A large part of category theory is the study of objects defined by universal properties.
The point is that an object is characterized not by the material from which it is built,
but by the maps into or out of it.

For instance, a product of objects $A$ and $B$ is an object $A\times B$
together with projections
\[
\pi_1\colon A\times B\to A,
\qquad
\pi_2\colon A\times B\to B,
\]
such that for every object $T$ and every pair of maps
\[
f\colon T\to A,
\qquad
g\colon T\to B,
\]
there exists a unique map
\[
\langle f,g\rangle\colon T\to A\times B
\]
with $\pi_1\circ \langle f,g\rangle = f$ and $\pi_2\circ \langle f,g\rangle = g$.
The relevant diagram is
\[
\begin{array}{ccccc}
& & T & & \\
& \swarrow f & \downarrow {\langle f,g\rangle} & \searrow g & \\
A & \xleftarrow{\pi_1} & A\times B & \xrightarrow{\pi_2} & B
\end{array}
\]
The dashed arrow is determined uniquely by the universal property.
This uniqueness is often more important than any explicit formula.

Another basic universal construction is the pullback.
Given $f\colon A\to C$ and $g\colon B\to C$,
a pullback is an object $A\times_C B$ fitting into a commutative square
\[
\begin{array}{ccc}
A\times_C B & \longrightarrow & B \\
\downarrow & & \downarrow g \\
A & \xrightarrow{f} & C
\end{array}
\]
that is universal among all such squares.
Pullbacks occur constantly in sheaf theory,
for example when one restricts local data to overlaps.

\subsection{Limits and colimits}

Two of the most basic families of universal constructions are \emph{limits} and \emph{colimits}.
A limit is a universal way of receiving compatible maps from an outside object into a diagram,
whereas a colimit is a universal way of sending compatible maps out of a diagram.
Products, pullbacks, and equalizers are examples of limits.
Coproducts, pushouts, and coequalizers are examples of colimits.

It is useful to begin with examples rather than the full abstract definition.
We have already seen that a product $A\times B$ is characterized by maps into $A$ and $B$.
This makes it a limit of the discrete diagram consisting of $A$ and $B$.
Similarly, the pullback
\[
A\times_C B
\]
is the limit of the cospan
\[
A \xrightarrow{f} C \xleftarrow{g} B.
\]
The pullback is the most basic way to enforce compatibility conditions.
In geometry and sheaf theory, it appears whenever one restricts two pieces of data to a common overlap.

Another important limit is the \emph{equalizer} of two parallel arrows
\[
A \rightrightarrows^{f}_{g} B.
\]
An equalizer is an object $E$ with a map $e\colon E\to A$ such that
\[
f\circ e = g\circ e,
\]
and which is universal with this property.
In \textbf{Sets}, one may think of it concretely as the subset
\[
E = \{a\in A : f(a)=g(a)\}.
\]
So the equalizer is the universal object on which the two maps become equal.
It is a limit because it receives maps from every object on which $f$ and $g$ already agree.

Dually, a coproduct $A\amalg B$ is characterized by inclusions
\[
A\to A\amalg B,
\qquad
B\to A\amalg B,
\]
such that maps out of $A\amalg B$ are exactly pairs of maps out of $A$ and $B$.
Likewise, a pushout is the colimit of a span
\[
A \xleftarrow{f} C \xrightarrow{g} B,
\]
and can be thought of as the universal way of gluing $A$ and $B$ along the common piece $C$.

Dually to the equalizer, one has the \emph{coequalizer} of two parallel arrows
\[
A \rightrightarrows^{f}_{g} B.
\]
A coequalizer is an object $q\colon B\to Q$ such that
\[
q\circ f = q\circ g,
\]
and which is universal with this property.
In \textbf{Sets}, one may think of it as the quotient of $B$ obtained by forcing
\[
f(a)\sim g(a)\qquad \text{for all } a\in A.
\]
So the coequalizer is the universal object in which the two maps become equal after passing to a quotient.
It is a colimit because every map out of $B$ that already identifies $f(a)$ and $g(a)$ factors uniquely through $Q$.

More generally, let $D\colon \mathcal J\to \mathcal C$ be a diagram.
A \emph{cone} from an object $L$ to $D$ is a family of maps
\[
L\to D(j)
\]
compatible with all arrows in $\mathcal J$.
A \emph{limit} of $D$ is a universal cone to $D$.
Dually, a \emph{cocone} from $D$ to an object $M$ is a family of maps
\[
D(j)\to M
\]
compatible with all arrows in $\mathcal J$,
and a \emph{colimit} of $D$ is a universal cocone.

In practice, one does not need the whole abstract definition at once.
What matters is to learn the logic of universal problems.
To construct a limit, one asks:
``What object receives compatible maps from everything else in the most universal way?''
To construct a colimit, one asks:
``What object is obtained by forcing compatible outgoing maps in the most universal way?''
This way of thinking is central in category theory and appears constantly in sheaf theory,
descent, and topos theory.

It is also helpful to pause over the traditional names \emph{projective limit}
(also called \emph{inverse limit}) and \emph{inductive limit}
(also called \emph{direct limit}).
An inverse system is typically a diagram
\[
X_0 \xleftarrow{} X_1 \xleftarrow{} X_2 \xleftarrow{} \cdots
\]
in which information is pushed backwards along transition maps.
Its limit is the object of all compatible families
\[
(x_0,x_1,x_2,\dots)
\]
with each $x_n$ matching the image of $x_{n+1}$.
So an inverse limit expresses the idea of keeping track of data at every stage
while enforcing perfect compatibility across all stages.
One may think of it as an object of coherent approximations.

Dually, a direct system is typically a diagram
\[
X_0 \xrightarrow{} X_1 \xrightarrow{} X_2 \xrightarrow{} \cdots
\]
in which information moves forward.
Its colimit identifies elements that become equal at some later stage.
Thus a direct limit expresses the idea of building a larger object by repeatedly adjoining data
and imposing the identifications generated by the transition maps.
One may think of it as the eventual object obtained from a process of accumulation.

In \textbf{Sets}, for example, the inverse limit of a sequence of quotient maps records sequences that remain compatible through every stage,
whereas the direct limit of a chain of inclusions is just the union of all stages.
These are excellent mental pictures to keep in mind:
inverse limits capture \emph{compatible families},
and direct limits capture \emph{eventual identification and accumulation}.

\subsection{Exercises on limits and colimits}

Because limits and colimits are often conceptually difficult at first,
it is helpful to practice them in concrete situations.

\begin{enumerate}[label=\textbf{Exercise C\arabic*.}, leftmargin=*]
\item Let $\mathcal C$ be a category with pullbacks.
Given arrows
\[
A \xrightarrow{f} C \xleftarrow{g} B,
\]
show that a morphism $X\to A\times_C B$ is equivalent to the data of morphisms
\[
u\colon X\to A,\qquad v\colon X\to B
\]
such that $f\circ u=g\circ v$.
Explain why this justifies the slogan that the pullback is the object of compatible pairs.

\item In the category of sets, let
\[
A=\{1,2\},\qquad B=\{a,b\},\qquad C=\{\ast\}
\]
with the unique maps $A\to C$ and $B\to C$.
Compute the pullback $A\times_C B$ explicitly.
Then compute the pushout of the span
\[
A \xleftarrow{\mathrm{id}_A} A \xrightarrow{h} B
\]
for a map $h\colon A\to B$ with $h(1)=a$ and $h(2)=b$.

\item Let $D\colon \mathcal J\to \mathcal C$ be a diagram and let $K\colon \mathcal J\to \mathbf 1$ be the unique functor to the terminal category.
Explain why a left Kan extension $\operatorname{Lan}_K D$ is the same thing as a colimit of $D$,
and why a right Kan extension $\operatorname{Ran}_K D$ is the same thing as a limit of $D$.
You do not need to prove the most general theorem formally;
it is enough to unpack the corresponding universal properties carefully.

\item Consider the inverse system in \textbf{Sets}
\[
\cdots \xrightarrow{\pi_4} \mathbb Z/16\mathbb Z \xrightarrow{\pi_3} \mathbb Z/8\mathbb Z \xrightarrow{\pi_2} \mathbb Z/4\mathbb Z \xrightarrow{\pi_1} \mathbb Z/2\mathbb Z,
\]
where each map is reduction modulo a smaller power of $2$.
Describe concretely what an element of the inverse limit of this system is.
Explain why this inverse limit should be thought of as a coherent system of approximations rather than as a single residue class modulo one fixed power of $2$.

\item Consider the direct system in \textbf{Sets}
\[
\{1,\dots,n\} \hookrightarrow \{1,\dots,n+1\} \hookrightarrow \{1,\dots,n+2\} \hookrightarrow \cdots,
\]
where each map is the evident inclusion.
Compute its colimit.
Then explain in words why this example captures the intuitive meaning of a direct limit as an ``eventual union'' of stages.
\end{enumerate}

\subsection{Kan extensions}

Kan extensions are among the most important general constructions in category theory.
They are, in a precise sense, the universal way of extending a functor along another functor.
Many familiar constructions are special cases of Kan extensions,
including several limits, colimits, sheafification-type procedures, and adjoints.

Suppose we are given categories $\mathcal A,\mathcal B,\mathcal C$,
a functor
\[
K\colon \mathcal A\to \mathcal B,
\]
and another functor
\[
F\colon \mathcal A\to \mathcal C.
\]
We would like to replace $F$ by a functor defined on $\mathcal B$.
In other words, we seek a functor
\[
G\colon \mathcal B\to \mathcal C
\]
which, in some universal sense, best approximates the idea that
\[
G\circ K \approx F.
\]
There are two versions of this problem.

A \emph{left Kan extension} of $F$ along $K$,
written
\[
\operatorname{Lan}_K F,
\]
is a functor $\mathcal B\to \mathcal C$
equipped with a natural transformation
\[
\eta\colon F \Rightarrow (\operatorname{Lan}_K F)\circ K
\]
which is universal among all such pairs.
This means that if $H\colon \mathcal B\to \mathcal C$ is any other functor
and $\alpha\colon F\Rightarrow H\circ K$ is any natural transformation,
then there exists a unique natural transformation
\[
\overline{\alpha}\colon \operatorname{Lan}_K F\Rightarrow H
\]
such that
\[
\alpha = (\overline{\alpha}K)\circ \eta.
\]

Dually, a \emph{right Kan extension} of $F$ along $K$,
written
\[
\operatorname{Ran}_K F,
\]
is a functor $\mathcal B\to \mathcal C$
equipped with a natural transformation
\[
\epsilon\colon (\operatorname{Ran}_K F)\circ K \Rightarrow F
\]
which is universal among all such pairs.
Thus right Kan extensions solve the best universal extension problem on the limiting side,
while left Kan extensions solve it on the colimiting side.

One reason Kan extensions matter is that they unify many definitions.
For example, if $K\colon \mathcal A\to \mathbf{1}$ is the unique functor to the terminal category,
then
\[
\operatorname{Lan}_K F
\]
is essentially the colimit of $F$,
while
\[
\operatorname{Ran}_K F
\]
is essentially the limit of $F$,
provided these exist.
So limits and colimits can themselves be regarded as Kan extensions to a point.

A second reason is that adjoints may be characterized by Kan extensions.
If $K\colon \mathcal A\to \mathcal B$ has a left adjoint,
then that left adjoint may be described as a suitable Kan extension of the identity functor.
Thus Kan extensions sit very near the center of category theory:
they encode the general pattern of extending structure universally.

For beginners, the slogan is:
\begin{quote}
A Kan extension is the universal best way to continue a functor beyond the category where it was originally defined.
\end{quote}
This slogan is not a definition, but it captures the right intuition.
Left Kan extensions are built from colimit-type information,
and right Kan extensions from limit-type information.
For that reason, they form a bridge between functoriality and the universal constructions discussed above.

\subsection{Representable functors and the Yoneda point of view}

Fix a category $\mathcal C$.
For any object $A$, one obtains a contravariant functor
\[
h_A := \Hom_{\mathcal C}(-,A)\colon \mathcal C^{\mathrm{op}}\to \Set,
\]
called the \emph{functor represented by $A$}.
It sends an object $X$ to the set of morphisms $X\to A$.
Morphisms into $A$ therefore define a functorial ``profile'' of $A$ as seen from the rest of the category.

The basic insight of Yoneda theory is that an object is completely determined by this profile.
That is, to understand $A$, it is often enough to understand all maps into $A$.
This principle is so important that it is worth stating in full.

\begin{theorem}[Yoneda lemma]
Let $\mathcal C$ be a locally small category,
let $A$ be an object of $\mathcal C$,
and let $F\colon \mathcal C^{\mathrm{op}}\to \Set$ be a presheaf.
Then there is a natural bijection
\[
\operatorname{Nat}(h_A,F) \cong F(A).
\]
Here $\operatorname{Nat}(h_A,F)$ denotes the set of natural transformations from $h_A$ to $F$.
\end{theorem}

\begin{proof}
Given a natural transformation $\eta\colon h_A\Rightarrow F$,
consider the element
\[
\eta_A(\id_A)\in F(A).
\]
This defines a function
\[
\Phi\colon \operatorname{Nat}(h_A,F)\to F(A),
\qquad
\Phi(\eta)=\eta_A(\id_A).
\]
Conversely, given $x\in F(A)$,
define for each object $X$ a map
\[
\eta^x_X\colon h_A(X)=\Hom_{\mathcal C}(X,A)\to F(X)
\]
by
\[
\eta^x_X(f)=F(f)(x).
\]
We must check that $\eta^x$ is natural.
If $u\colon Y\to X$,
then for any $f\colon X\to A$ we have
\[
\eta^x_Y(f\circ u)=F(f\circ u)(x)=F(u)(F(f)(x))=F(u)(\eta^x_X(f)),
\]
so the naturality square commutes.
Thus $x$ determines a natural transformation $\eta^x$.
This gives a function
\[
\Psi\colon F(A)\to \operatorname{Nat}(h_A,F),
\qquad
\Psi(x)=\eta^x.
\]
Now
\[
\Phi(\Psi(x))=\eta^x_A(\id_A)=F(\id_A)(x)=x,
\]
so $\Phi\circ \Psi=\id_{F(A)}$.
Conversely, if $\eta\colon h_A\Rightarrow F$,
then for each $f\colon X\to A$,
naturality applied to $f$ gives
\[
\eta_X(f)=F(f)(\eta_A(\id_A))=\bigl(\eta^{\eta_A(\id_A)}\bigr)_X(f).
\]
Hence $\Psi(\Phi(\eta))=\eta$.
So $\Phi$ and $\Psi$ are inverse bijections.
\end{proof}

A particularly important consequence is the \emph{Yoneda embedding}
\[
\mathcal C \hookrightarrow [\mathcal C^{\mathrm{op}},\Set],
\qquad
A\mapsto h_A,
\]
which is fully faithful.
Thus one may regard every category as living inside a category of set-valued functors.
This perspective lies in the background of sheaf theory and topos theory,
since a sheaf is precisely a presheaf satisfying an additional gluing condition.

\subsection{Cartesian closed categories}

A category $\mathcal C$ is called \emph{Cartesian closed}
if it has finite products and, for every pair of objects $A,B$,
an exponential object $B^A$.
The exponential is characterized by a natural bijection
\[
\Hom_{\mathcal C}(X\times A,B)
\cong
\Hom_{\mathcal C}(X,B^A)
\]
for all objects $X$.
Equivalently, maps out of $X\times A$ into $B$
are represented by maps out of $X$ into $B^A$.
This is the categorical form of currying.

In the category $\Set$, the exponential $B^A$ is the set of functions from $A$ to $B$.
The adjunction above says exactly that a function
\[
X\times A\to B
\]
is the same thing as a function
\[
X\to B^A.
\]
Thus Cartesian closed categories are categories in which internal function objects exist.

The reason CCCs matter here is that elementary topoi are, by definition,
finite-limit categories that are Cartesian closed and possess a subobject classifier.
So if one wishes to understand topoi,
one must become comfortable with the logic of products and exponentials.
The exponential object is what allows implication and function types to be represented internally.
This is one of the bridges from categorical structure to internal logic.

\subsection{Why these notions matter for the main text}

The main body of this paper uses all of the ideas above.
Sites and sheaves are formulated in categories.
Torsors on a site are functorial objects defined by local universal properties.
Topoi are categories with enough internal structure to support logic.
The Yoneda lemma lies silently behind the passage from objects to presheaves,
and Cartesian closedness lies behind internal function objects and implication.
Thus category theory is not merely background notation here.
It is the language in which the local-to-global structures of torsors, sheaves, and topoi become visible and manageable.

\section{Internal logic and intuitionistic logic}\label{app:intuitionistic}

The internal logic of an elementary topos is, in general, not classical but intuitionistic.
This statement deserves a careful explanation,
since it is one of the most important conceptual lessons of topos theory.

In classical logic, every proposition is either true or false,
and the law of excluded middle
\[
P\vee \neg P
\]
is always valid.
Equivalently, truth values are organized by a Boolean algebra.
In the category $\Set$, this is reflected in the fact that the subobject classifier is the two-element set $\{0,1\}$,
and subobjects of a set correspond to characteristic functions with Boolean values.

In a general topos, the object of truth values is $\Omega$.
The collection of its global or local truth values is not, in general, Boolean.
Rather, it carries the structure of a Heyting algebra.
A Heyting algebra still supports conjunction, disjunction, implication, and a notion of negation,
but negation is weaker than in Boolean logic.
The formula $P\vee \neg P$ need not hold.
For this reason, the internal logic is usually intuitionistic.

This is not a defect or an omission.
It is a faithful reflection of the fact that truth in a topos may be local.
In a sheaf topos $\Sh(X)$, a proposition over an open set $U$ may hold on one part of $U$ and fail on another.
Its truth value is therefore represented by an open subset of $U$.
Such local truth values naturally form a Heyting algebra of opens,
not in general a Boolean algebra.
Thus intuitionistic logic appears because the geometry of locality is itself non-Boolean.

One useful way to remember the difference is this:
classical logic asks whether a proposition has already been globally decided,
whereas intuitionistic logic asks whether one can construct or verify it in the available context.
In a topos, the available context is often local.
Thus intuitionistic logic is the natural logic of constructive local reasoning.

The relation with Kripke--Joyal semantics makes this especially transparent.
To say that $U\Vdash \exists x\,\varphi(x)$ means that one can pass to a cover of $U$ and find local witnesses there.
To say that $U\Vdash \forall x\,\varphi(x)$ means that the statement must continue to hold after every restriction to smaller opens.
These are precisely the kinds of rules one expects in intuitionistic semantics,
where truth is stable under refinement but may require local construction.

It is therefore reasonable to summarize the situation as follows.
Topos theory provides an ambient categorical setting in which geometric locality,
sheaf-theoretic gluing, and intuitionistic reasoning belong together.
The internal logic of a topos is intuitionistic because truth is mediated by context,
restriction, and compatibility rather than by a single global Boolean decision procedure.

This observation matters for the broader aims of the present paper.
The later cryptographic applications of~\cite{InoueSecurity}
are concerned with local observability, local consistency, controlled gluing of information,
and the disciplined passage from partial data to global structure.
For such purposes, intuitionistic logic is not merely philosophically interesting;
it is structurally appropriate.
It supplies a logical language in which the difference between global knowledge and locally forced information can be expressed without collapse.

\section{Solutions to the exercises}\label{app:solutions}

\paragraph{Solution to Exercise 1.}
For a sheaf topos $\Sh(X)$, the subobject classifier $\Omega$ assigns to each open set $U$ the set of open subsets of $U$.
Indeed, a subobject of the terminal sheaf restricted to $U$ is determined by the open region where it is inhabited,
so $\Omega(U)$ is naturally identified with the lattice of opens of $U$.
Hence a truth value over $U$ is not just a bit but a region of $U$ on which a statement holds.
This is why truth is local.

\paragraph{Solution to Exercise 2.}
Choose local sections $s_i\in \Pcal(U_i)$.
On each overlap $U_{ij}=U_i\times_U U_j$,
free transitivity gives a unique element $g_{ij}\in \G(U_{ij})$ such that
\[
s_j=s_i\cdot g_{ij}.
\]
On a triple overlap $U_{ijk}$, one may compare $s_k$ with $s_i$ directly and via $s_j$:
\[
s_k=s_i\cdot g_{ik}=s_j\cdot g_{jk}=s_i\cdot g_{ij}g_{jk}.
\]
By uniqueness of transport, $g_{ik}=g_{ij}g_{jk}$.
This is the cocycle condition.

\paragraph{Solution to Exercise 3.}
The statement $U\Vdash \exists x\,\varphi(x)$ means that after refining $U$ by a cover,
one can find local witnesses on the pieces.
Thus existence is allowed to be local.
By contrast, $U\Vdash \forall x\,\varphi(x)$ means that every section over every smaller open $V\subseteq U$ satisfies $\varphi$.
Universal truth must therefore survive every further localization.
This difference reflects the intuitionistic meaning of quantifiers.

\paragraph{Solution to Exercise 4.}
For $f\in C^0(U,\mathbb R)$,
the classifying map sends $f$ to the open subset
\[
\chi_P(f)=\{x\in U\mid f(x)>0\}.
\]
So the internal predicate ``$f>0$'' is interpreted as the open region where positivity holds.
This is a concrete model of a truth value in the sheaf topos.

\paragraph{Solution to Exercise 5.}
Finite limits provide products, terminal objects, pullbacks, and equalizers,
so one can form many familiar set-like constructions.
Exponentials provide internal function objects.
The subobject classifier provides an internal object of truth values and allows predicates to be represented categorically.
Among these ingredients, the subobject classifier is the one most directly responsible for internal truth values,
although the others are needed to support the broader logical and categorical structure.

\paragraph{Solution to Exercise C1.}
By the universal property of the pullback, to give a map
\[
\phi\colon X\to A\times_C B
\]
is equivalent to giving maps
\[
u=\pi_A\circ \phi\colon X\to A,\qquad v=\pi_B\circ \phi\colon X\to B
\]
such that
\[
f\circ u=g\circ v.
\]
Conversely, any such compatible pair $(u,v)$ induces a unique map
\[
X\to A\times_C B.
\]
Thus the pullback is precisely the universal object parameterizing compatible pairs of maps into $A$ and $B$ over $C$.

\paragraph{Solution to Exercise C2.}
Since both maps $A\to C$ and $B\to C$ are unique and $C$ has only one element, the compatibility condition is automatic.
Hence
\[
A\times_C B \cong A\times B
\]
as sets, so explicitly
\[
A\times_C B=\{(1,a),(1,b),(2,a),(2,b)\}.
\]

For the pushout of
\[
A \xleftarrow{\mathrm{id}_A} A \xrightarrow{h} B,
\]
one starts from the disjoint union $A\amalg B$ and identifies each element $x\in A$ on the left with its image $h(x)\in B$ on the right.
Since $h(1)=a$ and $h(2)=b$, the relations identify $1$ with $a$ and $2$ with $b$.
After imposing these identifications, no new points remain beyond the two classes represented by $a$ and $b$.
So the pushout is canonically isomorphic to $B$.

\paragraph{Solution to Exercise C3.}
A functor from the terminal category $\mathbf 1$ to $\mathcal C$ is just the choice of a single object of $\mathcal C$.
Thus a functor
\[
\operatorname{Lan}_K D\colon \mathbf 1\to \mathcal C
\]
is exactly an object $L$ of $\mathcal C$.
The natural transformation
\[
D \Rightarrow (\operatorname{Lan}_K D)\circ K
\]
is then precisely a cocone from the diagram $D$ to the object $L$.
The universal property of the left Kan extension says that this cocone is universal among all cocones out of $D$.
But that is exactly the definition of a colimit.

Dually, a right Kan extension
\[
\operatorname{Ran}_K D\colon \mathbf 1\to \mathcal C
\]
amounts to an object $R$ together with a natural transformation
\[
(\operatorname{Ran}_K D)\circ K \Rightarrow D,
\]
which is exactly a cone from $R$ to the diagram $D$.
Its universal property says that this cone is universal among all cones into $D$.
That is exactly the definition of a limit.

\paragraph{Solution to Exercise C4.}
An element of the inverse limit is a sequence
\[
(x_1,x_2,x_3,\dots),\qquad x_n\in \mathbb Z/2^n\mathbb Z,
\]
such that each $x_{n+1}$ reduces to $x_n$ modulo $2^n$.
So one is not choosing a single residue class once and for all;
one is choosing one residue class modulo $2$, one modulo $4$, one modulo $8$, and so on,
in such a way that all of them fit together compatibly.
This is why the inverse limit is best viewed as a coherent family of approximations.
Each stage gives finite information, and the limit remembers all stages simultaneously.
This is exactly the pattern that later leads to objects such as the $2$-adic integers.

\paragraph{Solution to Exercise C5.}
Since every map in the system is an inclusion,
no new identifications are imposed beyond those already present in the sets themselves.
Therefore the colimit is simply the union of all stages:
\[
\varinjlim_n \{1,\dots,n\} \cong \mathbb N_{>0}.
\]
This illustrates the intuitive meaning of a direct limit very well.
One starts with a sequence of objects that grow by adjoining new elements,
and the colimit is the eventual object containing everything that appears at some stage.
In this example, every positive integer appears after finitely many steps,
so the direct limit is the set of all positive integers.

So limits and colimits are Kan extensions to the terminal category:
\[
\operatorname{colim} D \simeq \operatorname{Lan}_K D,
\qquad
\operatorname{lim} D \simeq \operatorname{Ran}_K D.
\]

\section{A possible lecture plan}\label{app:lecture-plan}

This note can be taught in several ways, depending on the background of the audience.
If one aims at a serious but still introductory course for students who are not yet fully comfortable
with category theory, a natural format is a course of about \textbf{10--12 lectures}, each of length
\textbf{90 minutes}. A shorter course of 8 lectures is possible if some of the appendices are treated
as background reading, while a more leisurely 14-lecture version would allow more examples and exercises.
For most purposes, however, the 10--12 lecture format seems the most realistic.

The guiding principle of the course is to move from concrete local-to-global intuition to the more abstract
language of sheaves, topoi, and internal logic, and only then to explain why this matters for the
sheaf-theoretic approach to cryptographic security and $\Sigma$-protocols. In this sense, the course is
not merely a survey of topics, but a structured preparation for the paper
\emph{Grothendieck Topologies and Sheaf-Theoretic Foundations of Cryptographic Security:
Attacker Models and $\Sigma$-Protocols as the First Step}.

\subsection{Recommended overall length}

A reasonable recommendation is the following.
\begin{itemize}
\item \textbf{Minimum version:} 8 lectures of 90 minutes, if the appendices are assigned for self-study.
\item \textbf{Standard version:} 10 lectures of 90 minutes, for a balanced introductory course.
\item \textbf{Expanded version:} 12 lectures of 90 minutes, if one wants to spend time on exercises,
examples, and the logical aspects of topoi.
\end{itemize}

Among these, the \textbf{10-lecture plan} is probably the best default choice.
It is long enough to make the theory intelligible, but short enough to be used as a focused reading course
or graduate seminar.

\subsection{A standard 10-lecture plan}

We now describe a concrete 10-lecture plan.

\paragraph{Lecture 1. From torsors to local-to-global thinking.}
Review the basic idea of a torsor: free transitive action, lack of a preferred origin,
transport between points, and the role of local triviality.
Explain again why torsors naturally suggest cocycle data and gluing.
The goal of the lecture is to make clear that the passage to sheaves and topoi is not a change of subject,
but a continuation of the same local-to-global philosophy.

\paragraph{Lecture 2. Sites, Grothendieck topologies, and sheaves.}
Introduce categories as generalized indexing worlds, then define sieves and Grothendieck topologies.
Give concrete examples: open covers on a topological space and simple small-site examples.
Then explain presheaves and sheaves, emphasizing gluing and locality.
The pedagogical goal is to make the sheaf condition feel natural rather than formal.

\paragraph{Lecture 3. Torsors on a site, cocycles, and descent.}
Return to torsors, now in the language of sheaves of groups.
Explain local triviality on a site, Cech cocycles, and descent data.
This lecture is especially important because it ties the earlier torsor note to the present text
and makes precise the sense in which torsors already live naturally in a sheaf-theoretic world.

\paragraph{Lecture 4. Sheaf categories and Grothendieck topoi.}
Explain why categories of sheaves are not merely convenient containers, but genuine mathematical universes.
Introduce the slogan that a Grothendieck topos is a generalized space viewed through its sheaves.
Compare $\mathbf{Set}$ with $\mathbf{Sh}(X)$ and explain why the latter behaves like a variable version of the former.

\paragraph{Lecture 5. Elementary topoi.}
Introduce finite limits, exponentials, and subobject classifiers.
Present the definition of an elementary topos and explain why this definition abstracts the crucial structural
features of categories of sheaves.
It is helpful here to compare the notions with familiar constructions in $\mathbf{Set}$.

\paragraph{Lecture 6. Subobject classifiers and internal logic.}
Develop the idea that truth values in a topos may vary locally.
Explain the meaning of $\Omega$, characteristic morphisms, and the interpretation of predicates.
Introduce the basic intuition of internal logic and its local character.
If time permits, illustrate the discussion with a sheaf example where local truth and global truth differ.

\paragraph{Lecture 7. Intuitionistic logic and local reasoning.}
Build on the previous lecture and explain why the internal logic of a topos is generally intuitionistic.
Discuss conjunction, implication, and quantifiers in an informal but mathematically responsible way.
This is the place where students begin to see that logical structure is not external decoration,
but part of the geometry of a topos itself.

\paragraph{Lecture 8. Category-theoretic tools in practice.}
Use selected material from the appendix: Yoneda lemma, limits and colimits,
equalizers and coequalizers, pullbacks and pushouts, and the meaning of cartesian closedness.
Work through at least one exercise in class.
The goal is not to turn the course into a general category theory course,
but to provide enough operational fluency for the earlier lectures to become usable.

\paragraph{Lecture 9. Torsors and local symmetry inside a topos.}
Explain what it means to speak about group objects, torsors, and local symmetry internally to a topos.
This lecture should make clear that torsors are not left behind when one enters topos theory;
rather, they become more conceptually transparent.
Use this to revisit descent and local triviality from the internal point of view.

\paragraph{Lecture 10. Toward cryptographic security and $\Sigma$-protocols.}
Finally explain the conceptual bridge to the cryptographic paper.
Discuss why local consistency, simulability, witness data, and gluing are naturally illuminated by
sheaf-theoretic and topos-theoretic reasoning.
The aim is not yet to reproduce the full technical arguments of the cryptographic paper,
but to leave students in a position where that paper becomes meaningfully readable.

\subsection{A possible 12-lecture expansion}

If two additional lectures are available, a useful expansion is the following.
\begin{itemize}
\item Add one lecture entirely devoted to exercises on limits, colimits, Yoneda, and internal logic.
\item Add one lecture on reading selected passages of the cryptographic paper together with the present note,
so that the preparatory nature of this introduction becomes fully explicit.
\end{itemize}

This 12-lecture version is especially appropriate for a graduate seminar or supervised reading course.

\subsection{Pedagogical remarks}

The most difficult conceptual jump for beginners usually occurs at two points.
First, they may understand sheaves as a technical gadget without seeing why they lead naturally to topoi.
Second, they may understand the formal definition of a topos without seeing why internal logic matters.
For that reason, one should repeatedly return to a small number of central themes:
\begin{itemize}
\item local data and gluing,
\item absence of preferred global choices,
\item descent and compatibility,
\item local truth versus global truth,
\item the role of these ideas in the structural understanding of cryptographic security.
\end{itemize}

In short, this material is probably too rich for a very short mini-course,
but it is well suited to a \textbf{10--12 lecture course of 90 minutes each}.
Such a course would give students not only a first acquaintance with sheaves and topoi,
but also a serious conceptual preparation for the later study of $\Sigma$-protocols.


\begin{thebibliography}{99}

\bibitem{InoueTorsors}
T.~Inou\'e,
\emph{An Introduction to Torsors in Mathematics with a View Toward $\Sigma$-Protocols in Cryptography},
preprint, arXiv, 2026.

\bibitem{InoueSecurity}
T.~Inou\'e,
\emph{Grothendieck Topologies and Sheaf-Theoretic Foundations of Cryptographic Security:
Attacker Models and $\Sigma$-Protocols as the First Step},
preprint, arXiv, 2026.

\bibitem{MacLaneMoerdijk}
S.~Mac~Lane and I.~Moerdijk,
\emph{Sheaves in Geometry and Logic: A First Introduction to Topos Theory},
Springer, 1992.

\bibitem{JohnstoneSketches}
P.~T.~Johnstone,
\emph{Sketches of an Elephant: A Topos Theory Compendium, Vol.~1},
Oxford University Press, 2002.

\bibitem{JohnstoneElephant}
P.~T.~Johnstone,
\emph{Sketches of an Elephant: A Topos Theory Compendium, Vol.~2},
Oxford University Press, 2002.

\bibitem{SGA4}
M.~Artin, A.~Grothendieck, and J.~L.~Verdier,
\emph{Th\'eorie des topos et cohomologie \`etale des sch\'emas (SGA 4)},
Lecture Notes in Mathematics 269, 270, 305,
Springer, 1972--1973.

\bibitem{GoldwasserMR}
S.~Goldwasser, S.~Micali, and C.~Rackoff,
\emph{The Knowledge Complexity of Interactive Proof Systems},
SIAM Journal on Computing \textbf{18} (1989), 186--208.

\bibitem{DamgardSigma}
I.~Damg{\aa}rd,
\emph{On $\Sigma$-Protocols},
Lecture notes, Aarhus University, 2002.

\bibitem{Schnorr}
C.~P.~Schnorr,
\emph{Efficient Identification and Signatures for Smart Cards},
In: \emph{Advances in Cryptology -- CRYPTO '89},
Lecture Notes in Computer Science 435,
Springer, 1990, pp.~239--252.

\end{thebibliography}
\end{document}